\documentclass{article}
\usepackage{amssymb,amsmath,amsthm}
\usepackage{graphicx}
\usepackage[T1]{fontenc}
\usepackage{yfonts}

\newtheorem{theorem}{Theorem}

\newtheorem{lemma}{Lemma}
\newtheorem{proposition}{Proposition}

\newtheorem{remark}{Remark}
\newtheorem{corollary}{Corollary}
\newtheorem{definition}{Definition}
\newtheorem{example}{Example}

\newtheorem{problem}{Problem}

\def\bL{{\mathbb L}}
\def\bN{{\mathbb N}}

\def\bR{{\mathbb R}}

\def\cA{{\mathcal A}}

\def\cF{{\mathcal F}}
\def\cG{{\mathcal G}}

\def\cS{{\mathcal S}}
\def\cP{{\mathcal P}}
\def\cT{{\mathcal T}}

\def\cV{{\mathcal V}}

\def\fG{{\mathfrak G}}
\def\fH{{\mathfrak H}}

\def\fW{{\mathfrak W}}
\def\F{{_{F}}}

\def\G{{_{G}}}

\def\FG{{_{FG}}}

\def\Rt{{\bR^2}}

\def\vf{{\mathfrak X}_r(\Rt)}

\def\vfp{{\mathfrak X_r}(\Rt_\Phi)}

\def\smooth{{C^\infty(\bR^2)}}
\def\Cr{{C^r(\bR^2)}}

\def\smooth{C^\infty(\Rt)}

\title{Weak solutions of the cohomological equation on $\bR^2$ for regular vector fields}

\author{Roberto De Leo}

\begin{document}
\maketitle

{\sl Math Subject Classification:} Primary: 37J99, 53C12, 35F05

{\sl Keywords:} Cohomological equation, linear first-order PDEs, weak solutions

{\sl Address:} Dept. of Mathematics, Howard University, Washington, DC 20059 (USA) and
Istituto Nazionale di Fisica Nucleare, sez. di Cagliari, Monserrato (Italy)

{\sl Email:} roberto.deleo@howard.edu, roberto.deleo@ca.infn.it

\begin{abstract}
  In a recent article~\cite{DeL10c}, we studied the global solvability of the so-called 
  cohomological equation $L_\xi f=g$ in $C^\infty(\Rt)$, where $\xi$ is is a regular
  vector field on the plane and $L_\xi$ the corresponding Lie derivative.
  In a joint article with T. Gramchev and A. Kirilov~\cite{DGK10}, 
  we studied the existence of global $L^1_{loc}$ weak solutions of the 
  cohomological equation for vector fields depending only on one coordinate. 
  Here we generalize the results of both articles by providing explicit conditions 
  for the existence of global weak solutions to the cohomological equation when 
  $\xi$ is intrinsically Hamiltonian or of finite type.
\end{abstract}
\section{Introduction}
The topological structure of {\em regular} (i.e. without zeros) vector fields $\xi$ 
on $\Rt$ has been thoroughly investigated during the last century and is well 
understood (see~\cite{DeL10c} for detailed references).
The global analytic properties of the corresponding partial differential operators 
$L_\xi$ (Lie derivative in the $\xi$ direction) are, on the contrary, much 
less known. Some of their most basic properties were studied in~\cite{DeL10c}.

The main purposes of this article are to refine some of the 
results in~\cite{DeL10c}, concerning the action of $L_\xi$ on spaces of 
differentiable functions, and to use them to generalize to a much 
wider set of regular vector fields the results obtained in~\cite{DGK10}, 
concerning weak solutions of the cohomological equation
%
\begin{equation}
  \label{eq:ce}
  L_\xi f=g\in C^k(\Rt)
\end{equation}
for regular planar vector fields depending only on one variable.
%
\section{Definitions and main results}
The following definitions and notations will be used in the present article. 

\noindent{\bf Vector Fields and Foliations.}
We will usually denote vector fields by $\xi$ and, to avoid ambiguities,
the corresponding Lie derivative operator by $L_\xi$.
We say that a $C^1$ function is \emph{regular} if its differential
is never zero; analogously, we say that a vector field is regular when it
has no zeros.
We denote by $\vf$ the set of all smooth regular vector fields on the plane.
Given any $\xi\in\vf$, $\cF_\xi$ will denote the smooth foliation of its integral 
trajectories and, by abuse of notation, the space of leaves endowed with its 
canonical quotient smooth structure\footnote{The spaces of leaves are often 
non-Haussdorf. For an introduction to non-Haussdorf smooth structures, see~\cite{HR57}}.
We denote by $\pi_\xi:\Rt\to\cF_\xi$ the canonical projection that sends a point to 
the leaf passing through it. A {\em saturated} neighborhood of a leaf $\ell$ of 
$\cF_\xi$ is a set $\pi_\xi^{-1}(U)$, where $U$ is a neighborhood of $\ell$ in $\cF_\xi$.
We say that two integral trajectories $s_1$, $s_2$ of $\xi$ are \emph{inseparable}
when they are inseparable as points in the topology of $\cF_\xi$
(e.g. see Fig.~\ref{fig:hstrip}).
An integral trajectory $s$ which is inseparable from some other integral trajectory
is said a {\em separatrix}. We denote by $\cS_\xi$ the set of all separatrices 
of $\xi$.
\begin{definition}
  A regular planar vector field $\xi$ (and, by extension, the foliation $\cF_\xi$) 
  is of \emph{finite type} 
  if the set of its separatrices is closed in $\cF_\xi$
  and every separatrix is inseparable from just finitely many other integral trajectories. 
\end{definition}
The set of vector fields of finite type is of great relevance
since important categories of vector fields belong to it. 
For example, every regular polynomial vector field is of finite type: finite bounds 
for the number of separatrices of a polynomial vector field were found 
first by Markus~\cite{Mar72} and later improved independently by 
M.P.~Muller~\cite{Mul76b} and S.~Schecter and M.F. Singer~\cite{SS80}. 
It is easy to verify that also all vector fields strongly proportional to
regular vector fields invariant with respect to translations in  
a given direction are of finite type.
An important feature of vector fields of finite type is that the complement 
of the set of the separatrices of a vector field
of finite type is the disjoint union of countably many unbounded connected open sets 
(named by Markus~\cite{Mar54} \emph{canonical regions}) whose boundary 
has only a finite number of connected components.

We say that a smooth foliation is \emph{Hamiltonian} if its leaves are the 
level sets of a regular smooth function. 
A vector field $\xi\in\vf$ is Hamiltonian when $L_\xi(dx\wedge dy)=0$.
\begin{definition}
  Two smooth vector fields are \emph{strongly proportional} if they
  are proportional through a strictly positive or strictly negative smooth function.
  A smooth vector field is \emph{intrinsically Hamiltonian} if it is strongly
  proportional to a smooth Hamiltonian vector field and is 
  \emph{transversally Hamiltonian} if it is transversal to a Hamiltonian smooth
  vector field (equivalently, to the level sets of a regular smooth function). 
\end{definition}
It is easily seen that a smooth regular vector field $\xi$ is intrinsically Hamiltonian 
if and only if the partial differential equation $L_\xi f=0$ admits a smooth regular 
solution and is transversally Hamiltonian if and only if 
the partial differential inequality $L_\xi f>0$ has a smooth solution. 
It was proved in~\cite{Wei88} (resp. in~\cite{DeL10c}) that every Hamiltonian 
vector field (resp. every vector field of finite type) is transversally Hamiltonian.
%

\noindent {\bf Differential Structures.} Every locally injective 
homeomorphism $\Phi:\Rt\to\bR$ determines a differential $C^\infty$ structure
on $\Rt$ given by the atlas containing all charts $(U,\Phi|_U)$, where 
$U\subset\Rt$ is any open set on which $\Phi$ is injective. We denote by
$\Rt_\Phi$ the differential manifold $\Rt$ with the differential structure
induced by $\Phi$. In general $\Rt_\Phi$ is different from $\bR$, and in particular
$C^\infty(\Rt_\Phi)\neq C^\infty(\Rt)$, but $\Rt_\Phi$ is always globally 
diffeomorphic to $\Rt$, e.g. because it is well-known that, in dimension
smaller than four, every topological manifold admits just one differential
structure modulo diffeomorphisms (see~\cite{Moi77} and~\cite{Rud01}).
%

\noindent{\bf Functional Spaces.} 
Let $\bL^1_0=[-1,0)$ and 
$\bL^2_0=[-1,0]\times[-1,1]\setminus(0,0)$. 
We denote by $\fH^k(\bL^i_0)$, $i=1,2$, the ring of left germs 
at the origin of functions in $C^k(\bL^i_0)$, i.e. the equivalence classes 
determined by the equivalence relation $h\simeq h'$ if $h$ and $h'$ coincide 
in some left neighborhood of the origin. 
We focus our attention of the following subrings of $\fH^k(\bL^i_0)$:
$\fG^{r,k}(\bL^i_0)$, with $r=0,\dots,k$, containing the left germs 
of all functions belonging to $C^r(\overline{\bL}^i_0)\cap C^k(\bL^i_0)$,
and $\fW^{l,p,k}(\bL^i_0)$, with $p\geq1$ and $l=0,\dots,k+3$, 
containing the left germs of all 
functions belonging to $W^{l,p}_{loc}(\bL^i_0)\cap C^k(\bL^i_0)$.
%
\begin{definition}
  We call \emph{singular left germs at the origin} the elements 
  of the quotient rings $S\fG^{r,k}(\bL^i_0)=\fH^k(\bL^i_0)/\fG^{r,k}(\bL^i_0)$ 
  and $S\fW^{l,p,k}(\bL^i_0)=\fH^k(\bL^i_0)/\fW^{l,p,k}(\bL^i_0)$.
  By abuse of notation, we denote by $S\fG^{k+1,k}(\bL^2_0)$
  the singular germs of germs of $C^k$ functions which are $C^{k+1}$ in
  the first variable. 
\end{definition}

\noindent{\bf Main objectives.} 
We consider the partial differential operators\footnote{We will omit the
upper index in the operators $L_\xi$ when there is no ambiguity.}
$$
L^{(r)}_\xi:C^r(\Rt)\to C^{r-1}(\Rt)\,,\;r=1,2,\dots,\infty
$$ 
and their weak extensions 
$$
L^{(l,p)}_\xi:W^{l,p}_{loc}(\Rt)\to W^{l-1,p}_{loc}(\Rt)\,,\;p\geq1,\;l=1,2,\dots\,,
$$
where $W^{l,p}_{loc}(\Rt)$ is the Sobolev space of $L^p_{loc}$ functions whose first 
$l$ weak derivatives 
are also $L^p_{loc}$. We endow $C^r(\Rt)$ with the Whitney topology.
Our main aim is studying the images
$$
L_\xi\left(C^r(\Rt)\right)\cap C^k(\Rt)\,,\;
L_\xi\left(W^{l,p}_{loc}(\Rt)\right)\cap C^k(\Rt).
$$
In other words, we study the existence of global $C^r$ or $W^{l,p}_{loc}$
solutions of the cohomological equation when the right hand side
is of class $C^k$ (notice that, as easily shown via the method of characteristics, 
such solutions are at least $C^k$ everywhere except, at most, on the separatrices).
In case there is no regularity loss (i.e. $r=k+1$),
the problem reduces to studying the full image $L_\xi(C^{k+1}(\Rt))$.

We point out that, since clearly 
$$
L_\xi(C^{k+1+k'}(\Rt))\cap C^k(\Rt)=L_\xi(C^{k+1+k'}(\Rt))\cap C^{k+k'}(\Rt)\hbox{ for all $k'\geq1$},
$$ 
it is enough to consider, for any given $k$, just the cases $r=1,\dots,k+1$.
Similarly, since $W^{k+1,p}(\Rt)\subset C^{k}(\Rt)$ for $p>2$ and 
$W^{k+2,p}(\Rt)\subset C^{k}(\Rt)$ for $1\leq p\leq2$ (e.g. see~\cite{DD12}),
it is enough to consider the cases $l=0,\dots,k+1$ for $p>2$ and
$l=0,\dots,k+2$ for $1\leq p\leq2$.
\begin{remark}
  Note that it makes sense considering the case $l=0$, namely 
  $W^{0,p}_{loc}=L^p_{loc}$ solutions, because, as noted above, all solutions 
  are always of class at least $C^k$ outside of the separatrices
  and, by definition, their derivative in the direction $\xi$ is also $C^k$.
  Namely, ``$L^p_{loc}$'' weak solutions of $L_\xi f=g\in C^k(\Rt)$ are 
  actually elements of $L^p_{loc}(\Rt)\cap C^k(\Rt\setminus\cS_\xi)$.
\end{remark}
%
%

%
\begin{example}
  \label{ex:motivations}
  Consider the regular vector field $\xi=2y\partial_x+(1-y^2)\partial_y$. 
  As a corollary of Proposition 2 in~\cite{DeL10c}, 
  $1\not\in L_\xi(\Cr)\cap C^\infty(\Rt)$
  for any $r$. On the other side 
  $1\in L_\xi(L^{1}_{loc}(\Rt))\bigcap C^\infty(\Rt)$
  since, for example, 
  $L_\xi f=1$ for $f(x,y)=\frac{1}{2}\ln\left|\frac{1+y}{1-y}\right|$. 
  Note that solutions that diverge on just one of the two separatrices
  can be easily obtained through the $L_\xi$'s weak first-integral
  $h(x,y)=x+\ln|1-y^2|$.
\end{example}
%

\noindent{\bf Main Results.}
In Section~\ref{sec:geometry}
we refine some results in~\cite{DeL10c} about the local geometry of Hamiltonian
and finite type regular foliations on the plane. The section's main result, 
contained in Theorem~\ref{thm:NC}, is that for such vector fields, locally, the problem 
of the extension of a solution of the cohomological equation from a saturated 
neighborhood of a separatrix $s_1$ to the saturated neighborhood of an adjacent 
separatrix $s_2$ can be always reduced to the problem of the extension, from 
$\bL^2_0\setminus\{0\}\times[0,\infty)$ to the whole $\bL^2_0$, of a solution 
of $\partial_y f=g\in C^k(\bL^2_0)$. This theorem generalizes Proposition~8 
in~\cite{DeL10c} and fixes a minor mistake in its statement.

In Section~\ref{sec:ce} we study the images of $L^{(r)}_\xi$ and $L^{(l,p)}_\xi$
Our main results are contained in Theorem~\ref{thm:Germs}, where we provide 
explicit criteria for the solubility of the cohomological equation in the
Hamiltonian case, and Theorem~\ref{thm:GermsNH}, were weaker results are
provided for the finite-type non-Hamiltonian case.
Moreover, in Theorem~\ref{thm:OpCl}, we show that the solvability of the
cohomological equation, in the Hamiltonian case, is stable with respect 
to small perturbation of the right hand side.

Finally, in Section~\ref{sec:examples}, we present in some detail four concrete
examples. In the first two we consider, respectively, the cases of two regular 
Hamiltonian and non-Hamiltonian vector fields depending only on one variable 
and with just a pair a separatrices and compare our results with those in~\cite{DGK10}. 
In the last two we consider two case not covered by the results in~\cite{DGK10}:
the case of a regular Hamiltonian vector field with just a pair a 
separatrices and not invariant with respect to translations in any direction 
and the case of a regular Hamiltonian vector field with three separatrices
inseparable from each other.
\section{Geometry of $\cF_\xi$}
\label{sec:geometry}
%
%
The geometry of the set of separatrices of a regular planar foliation
can be quite non-trivial.
Explicit examples of $C^\infty$ (see~\cite{Waz34} and~\cite{Wei88}) and 
$C^\omega$ (see~\cite{Mul76a}) foliations of the plane with a set 
of separatrices dense on some open set are known in literature. 
In the example below, we show how to build an explicit instance, 
simpler and more natural than the ones mentioned above,
of $C^\infty$ foliation whose separatrices are dense on the whole plane.

%
\begin{example}
\label{ex:dense}
  The building blocks of the present example are the $C^\infty$ foliation 
  $\cS$ of the vertical half-stripe $S=[0,1]\times[0,\infty)$, shown in 
  Fig.~\ref{fig:hstrip} (left),
  and the foliation $\cF_0$ in vertical lines of the whole plane. 

  Denote by $t$ the $x$-axis, everywhere transversal to $\cF_0$, and by $P_k$ 
  the half-plane $y<k$ and select a sequence of vertical half-stripes 
  $S_n=[\ell_n,\ell'_n]\times[n,\infty)$ such that
  $S_i\cap S_j=\emptyset$ for $i\neq j$.
  After replacing, on each $S_n$, the vertical foliation 
%
  with a suitably rescaled version of $\cS$, we get
  a $C^\infty$ foliation of the plane $\cF_1$ coinciding with $\cF_0$ 
  for $y<1$ and with a set of separatrices dense on the open set 
  $U=\pi_{\cF_1}^{-1}(\pi_{\cF_1}(t))\supset P_1$. 
  
  Now, let $t_n$ be the image of the transversal $\gamma_n$ in $S_n$,
  $U_n=\pi_{\cF_1}^{-1}(\pi_{\cF_1}(t_n))$ and let 
  $\Phi_n=(\varphi_n,\psi_n):U_n\to\Rt$ be a rectifying diffeomorphism
  such that:
  1. $t_n$ has equation $\psi_n=0$; 2. the leaves of 
  $U_n$ are sent to vertical lines and, in particular, 
  $s_n$ has equation $\varphi_n=0$; 3. the leaves outside $U$
  are those for which $\varphi_n>0$. With the same construction described
  above, we can modify this foliation in the half-plane $\varphi_n>0$
  to produce a new foliation $\cF_2$ which is dense on the open set
  $U\bigcup_{n\in\bN}\pi_{\cF_2}^{-1}(\pi_{\cF_2}(t_n))$
  and coincides with $\cF_1$ on $P_2$.

  By repeating this construction recursively, we get a sequence $\{\cF_{n}\}$ 
  of $C^\infty$ regular foliations of the plane such that
  $\cF_{n+1}|_{P_{n+1}}=\cF_{n}|_{P_{n+1}}$ for every $n\in\bN$
  and the closure of the set of the separatrices of $\cF_n$ contains $P_n$. 
  Hence the $\cF_n$ converge to a smooth foliation $\cF_\infty$ 
  with a set of separatrices dense in the whole plane.

\end{example}
\begin{remark}
  The leaf space of $\cF_\infty$ 
  is a 1-dimensional, smooth non-Hausdorff smooth
  manifold with a dense set of \emph{binary} branch points, namely
  such that at every branch point exactly two branches (or \emph{plumes})
  meet. It is easy to modify the foliation $\cS$ in order to have, 
  at every branch point, the concurrence of any finite number of plumes, 
  or even infinitely (countably) many.
\end{remark}
%
Note that the foliation $\cS$ (see Fig.~\ref{fig:hstrip}) is Hamiltonian, 
so that also every $\cF_n$, and therefore even $\cF_\infty$, is Hamiltonian.
In particular, as a corollary of a Lemma of Weiner~\cite{Wei88} stating that 
the first component projection $\pi:Imm^\infty(\Rt,\Rt)\to Sub^\infty(\Rt,\bR)$ 
is surjective,
this shows that the topology of immersions of the plane into itself can be
quite non-trivial:
%
\begin{proposition}
  For every $k=2,3,\dots,\infty$ there exists a $C^\infty$ immersion
  $\Phi_\FG=(F,G):\Rt\to\Rt$ such that
  the space of the foliation $\cF$ 
  of the level sets of $F$ is a infinite feather of order $k$.
\end{proposition}
At the other end of the spectrum,
if we use instead the foliation $\cS'$ in the construction presented in
Example~\ref{ex:dense}, we end up with a foliation $\cF'_\infty$ which cannot 
be obtained as the level set of any
$C^1$ function. Indeed, by construction, any $C^1$ function giving rise to
the foliation of $\cS'$ has differential equal to 0 on $x=-1$.
Since the final foliation was built in such a way that the set of separatrices 
on which the differential is zero is dense on the plane, any $C^1$ function 
having those lines as level sets must have its differentiable null on a
dense set and therefore, by continuity, must be constant.
%
%
\begin{figure}
  \label{fig:hstrip}
  \centering
  \includegraphics[width=5.5cm]{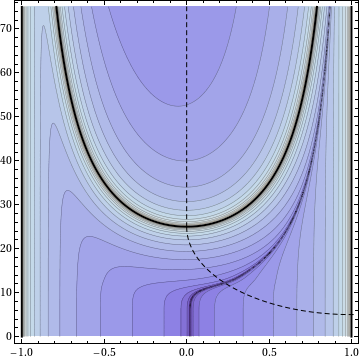}\hskip.5cm
  \includegraphics[width=5.5cm]{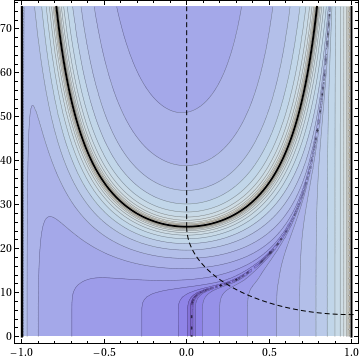}
  \caption{%
    \footnotesize 
    Two $C^\infty$ foliations $\cS$ (left) and $\cS'$ (right) of the strip 
    $S=[0,1]\times[0,\infty)$.
    [left] $\cS$ is Hamiltonian and it can be obtained as 
    the level sets of the (non-regular) $C^\infty$ function
    $F(x,y)=\arctan\left[\frac{1}{\left(f\left(\frac{15-x}{10}\right)h(5) 
      + \left(1 - f\left(\frac{15-x}{10}\right)\right)h(x) 
      + y^2 - 1\right)^3} + \frac{1}{(y - 1)^3} + \frac{1}{(y + 1)^3}\right]^2$,
    where $f(x)$ is any $C^\infty$ function equal to $0$ for $x<0$, 1 for $x>1$
    and strictly increasing between $0$ and 1, and $h(x)=\frac{30}{x+5}$.
    The black thick lines $s_1=\{x=-1\}$ and $s_2=\{y=5\frac{5+x^2}{1-x^2}\}$ 
    are inseparable in the quotient topology.
    The black dashed line $\gamma$ is a particular choice of a globally 
    transverse to the foliation.
    [right] $\cS'$ has the same topology as $\cS$ but it is non-Hamiltonian and 
    it can be obtained as the level sets of the non-regular $C^\infty$ function
    $F(x,y)=\arctan\left[\frac{1}{f\left(\frac{15-x}{10}\right)h(5) 
      + \left(1 - f\left(\frac{15-x}{10}\right)\right)h(x) 
      + y^2 - 1} + \frac{1}{(y - 1)^3} + \frac{1}{(y + 1)^3}\right]^2$.
    Its two separatrices coincide with the ones of the previous foliation.
  }
\end{figure}
%

At this regard, it is important to notice that, for a foliation in $\Rt$,
the lack of a smooth 
regular first-integral is a relative, rather than absolute,
property. Indeed, as a corollary of a result of Kaplan that \emph{every} 
regular foliation of the plane has a \emph{continuous} first-integral~\cite{Kap40}, 
we can prove the following:
%
\begin{theorem}
  Every regular $C^0$ foliation of the plane is a regular $C^\infty$ Hamiltonian 
  foliation for some suitable differential structure.
\end{theorem}
\begin{proof}
  Let $\cF$ be a regular planar foliation and let $\cG$ be any foliation
  everywhere transversal to it. Assume that $\cF$ is not Hamiltonian, 
  otherwise there is nothing to prove, and let $F$ and $G$ two continuous 
  first-integrals of, respectively, $\cF$ and $\cG$. 
  Then the map $\Phi_\FG=(F,G):\Rt\to\bR$ is locally injective and,
  in $\Rt_{\Phi_\FG}$, both $\cF$ and $\cG$ are $C^\infty$ Hamiltonian.
  Indeed in every chart $(U,(\Phi_\FG)|_U)$, by definition, $F$ and $G$ are, 
  respectively, the $x$ and $y$ coordinates and therefore are $C^\infty$ 
  and their differential is nowhere zero. Moreover, by construction, 
  $\cF=\{dF=0\}$ and $\cG=\{dG=0\}$. 
\end{proof}
%
Applying to this context the proof of Weiner's Lemma in~\cite{Wei88} and of
Theorem~2 in~\cite{DeL10c}, we can prove the following more general result:
\begin{theorem}
  Let $\cF$ be a $C^0$ foliation of $\Rt$. Then,
  if $\cF$ is $C^r$ and either Hamiltonian or of finite type 
  in $\Rt_\Phi$,
  it admits a transversal $C^r$ Hamiltonian foliation in $\Rt_\Phi$.
\end{theorem}
%
%

Now we refine some result in~\cite{DeL10c} on the local geometry of
regular foliations.
%
\begin{definition}
  Given a pair of adjacent inseparable leaves $s_1,s_2$ of a foliation
  $\cF$, we say that a curve $\gamma$ {\em separates} them, or that $\gamma$
  is \emph{between} $s_1$ and $s_2$, if $s_1$ and
  $s_2$ belong to different connected components of $\Rt\setminus\gamma$.
  We say that a foliation $\cG$ transversal to $\cF$ \emph{minimally separates}
  $\cF$ if there is only one leaf of $\cG$ between every two adjacent separatrices
  of $\cF$.
\end{definition}
\begin{figure}
  \label{fig:reduction}
  \centering
  \includegraphics[width=5.5cm]{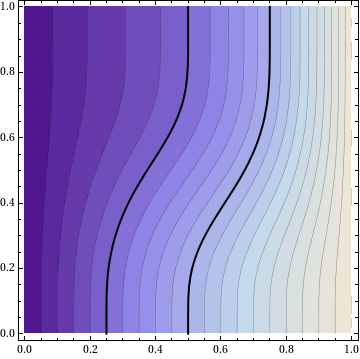}
  \caption{%
    \footnotesize Basic block to eliminate all leaves (except one)
    separating two consecutive inseparable leaves of a foliation.
  }
\end{figure}
We start with a technical Lemma:
\begin{lemma}
  \label{thm:tech}
  Consider the foliation $\cV$ in vertical lines of the set
  $$S=(-1,0]\times(-2,2)\setminus\{0\}\times[-1/2,1/2].$$
  There exists a $C^\infty$ Hamiltonian foliation $\cT$ of $S$ which 
  is everywhere transversal to $\cV$ and minimally separates it.
\end{lemma}
\begin{proof}
  We start by replacing the rectangle $R_1=[1/2,1/4]\times[-1,0]$ with a suitably 
  rescaled copy of the rectangle $R$ shown in Fig.~\ref{fig:reduction}. Because
  of how $R$ is foliated, after the substitution we get a foliation with the 
  same regularity of the previous one and still transversal at every point
  to the vertical direction.  
  Inside $R_1$ though, all leaves crossing the line $x=-1$ at $-1/2\leq y\leq -1/4$ 
  are rerouted so that they cross $\ell_1$ at $-3/4\leq y\leq -1/2$.
  
  Now we repeat the procedure by replacing the rectangles 
  $R_n=[2^{-n},2^{-n-1}]\times[-1,0]$ with suitably rescaled copies of $R$.
  At the step $n$ therefore we still have a foliation $\cG_n$ of the same regularity
  of the original one and everywhere transversal to the vertical direction
  but this time though, among all leaves crossing $x=-1$ for $-1\leq y\leq0$,
  only those for which $-1/2^{n+1}< y\leq0$ do not cross $\ell_1$.

  In the limit for $n\to\infty$ we are left with a foliation $\cG'$ such 
  that all of its leaves crossing $x=-1$ for $-1\leq y<0$ do cross the line 
  $\ell_1$. By replacing the upper half of the foliation with the symmetric 
  of the lower part with respect to the $x$ axis we found the foliation 
  $\cG$ mentioned above.

  The claim is proved by repeating this procedure for all pairs of adjacent
  separatrices of the foliation.
\end{proof}
\begin{theorem}
  \label{thm:minsep}
  Let $\cF$ be a $C^0$ foliation of $\Rt$. Then
  there exists a $C^0$ transverse foliation $\cG$ which minimally 
  separates $\cF$.
  Moreover, if $\cF$ is $C^r$ and either Hamiltonian or of finite type 
  in some $\Rt_\Phi$,
  then $\cG$ can be chosen to be $C^r$ and Hamiltonian with respect to $\Rt_\Phi$.
\end{theorem}
\begin{proof}
  Let $\cA$ be any atlas of $\Rt$ where $\cF$ is of class $C^r$.
  Then, either by Weiner's Lemma in~\cite{Wei88} (if $\cF$ is Hamiltonian)
  or by Theorem~2 in~\cite{DeL10c} (if it is of finite type), there exists
  a $C^r$ locally injective map $\Phi_\FG=(F,G)$ (which is an immersion if
  $\cF$ is Hamiltonian) that sends $\cF$ and $\cG$, respectively, 
  in vertical and horizontal lines. 

  By definition of inseparable leaves, for every pair of adjacent
  inseparable leaves $s_1,s_2\in\cF$ cut, respectively, by the
  transversals $t_1,t_2\in\cG$, the set 
  $U_{12}=\pi_\cF^{-1}(\pi_\cF(t_1))\cap\pi_\cF^{-1}(\pi_\cF(t_2))$ contains
  a saturated one-sided (left or right, in the $(F,G)$ coordinates) 
  neighborhood of $s_1$ and $s_2$. Let $s_1\cup s_2\subset F^{-1}(a)$,
  $t_i=G^{-1}(b_i)$, with $b_1<b_2$, and assume, for the sake of the 
  argument, that $U_{12}\subset F^{-1}((-\infty,a))$. 
  Then $U_{12}\supset R_\epsilon=(a-\epsilon,a)\times(b_1,b_2)$ for some 
  $\epsilon>0$.

  By Lemma~\ref{thm:tech}, we can replace the restriction of $\cG$ to
  $R_\epsilon$ with a new foliation in such a way that the new foliation
  $\cG'$ is still Hamiltonian, has the same regularity of $\cG$ and 
  separates minimally $s_1$ and $s_2$. The proof is concluded by repeating 
  this process for all pairs of adjacent separatrices of $\cF$.
\end{proof}
%
%
The previous result allows us to state a stronger version of Proposition~8 
in~\cite{DeL10c}.
This version also fixes a minor mistake in that Proposition's claim.
\begin{theorem}
  \label{thm:NC}
  Let $\xi\in\vfp$ be either Hamiltonian or locally finite
  and let $F\in C^\infty(\Rt_\Phi)$ be a generator of $\ker L_\xi$ and 
  $G\in C^\infty(\Rt_\Phi)$ such that $L_\xi G>0$ and $\cG=\{dG=0\}$ minimally 
  separates $\cF_\xi$.
  Then, for every pair of adjacent separatrices $s_1,s_2\in\cF_\xi$, with 
  $s_1\cup s_2\subset F^{-1}(a)$, separated by $t\in\cG$, with $t\subset G^{-1}(b)$,
  and leaves $t_1,t_2\in\cG$, with $t_i\subset G^{-1}(b_i)$, cutting respectively
  $s_1$ and $s_2$, set $U=\pi_\xi^{-1}(\pi_\xi(t_1))\cup\pi_\xi^{-1}(\pi_\xi(t_2))$
  and $V=\pi_\xi^{-1}(\pi_\xi(t_1))\cap\pi_\xi^{-1}(\pi_\xi(t_2))$.
  The map $\Phi_\FG=(F,G):\Rt\to\Rt$ satisfies the following conditions:
  \begin{enumerate}
  \item the restriction of $\Phi_\FG$ to $U$ is a diffeomorphism onto $\Phi_\FG(U)$;
  \item the leaves of $\cF_\xi$ and $\cG$ are mapped, respectively, 
    into vertical and horizontal lines;
  \item 
    $\Phi_\FG(s_i)=\{a\}\times G(s_i)$, with 
    $G(s_1)\cup G(s_2)=(b'_1,b)\cup(b,b'_2)$ 
    for some $-\infty\leq b'_1<b$ and $b<b'_2\leq\infty$;
  \item if $V\subset F^{-1}((-\infty,a))$
    (resp. if $V\subset F^{-1}((a,\infty))$,
    then $\Phi_\FG(t)=(a_1,a)\times\{b\}$ for some $-\infty\leq a_1<a$
 (resp. $\Phi_\FG(t)=(a,a_1)\times\{b\}$ for some $a<a_1\leq\infty$)
    and $(a-\epsilon,a)\times(b_1,b_2)\subset\Phi_\FG(V)$ 
 (resp. $(a,a+\epsilon)\times(b_1,b_2)\subset\Phi_\FG(V)$) for some $\epsilon>0$.
  \end{enumerate}
\end{theorem}
\begin{definition}
  Under the conditions of the previous theorem, we call $\Phi_\FG:U\to\Rt$ 
  a \emph{normal chart} for the adjacent separatrices $s_1,s_2$.
\end{definition}
\section{$L_\xi(C^r(\Rt))\cap C^k(\Rt)$ and $L_\xi(W^{l,p}_{loc}(\Rt))\cap C^k(\Rt)$}
\label{sec:ce}
It is well known that local solutions to the cohomological equation~(\ref{eq:ce})
can be built through the method of characteristics. When $\xi\in\vf$ and
$g\in C^k(\Rt)$, the only obstruction to the existence of a $C^r$ solution, 
$0\leq r\leq k$, is the problem of extending a local 
solution across pairs of adjacent separatrices (e.g. see~\cite{DeL10c}). 
In a normal chart (see Theorem~\ref{thm:NC}), this problem can always be reduced 
to the following:
%
\begin{problem}
  Let $g\in C^k(\bL^2_0)$ and $\varphi\in C^r(\overline{\bL^1_0}\times\{-1\})$
  and define on $\bL^1_0\times\{1\}$ the function 
  $\psi(x)=\int_{-1}^1 g(x,t)dt+\varphi(x)$. 
  Under which conditions on $g$ the function $\psi$ can be extended 
  to a $C^r$ function at $x=0$? 
  Similarly, assuming 
  $\varphi\in W^{l,p}(\bL^1_0\times\{-1\})\cap C^r(\bL^1_0\times\{-1\})$, 
  under which conditions on $g$ the function $\psi$ does is also 
  $W^{l,p}_{loc}$ at $x=0$?
\end{problem}
%
%
\begin{definition}
  We denote, respectively, by $\theta_{r,k}:S\fG^{k,r}(\bL^2_0)\to S\fG^{k,r}(\bL^1_0)$ 
  and $\theta_{l,p,k}:S\fW^{l,p,k}(\bL^2_0)\to S\fW^{l,p,k}(\bL^1_0)$ 
  the homomorphisms associating, to the left singular germ at $(0,0)$ 
  of a function $g\in C^k(\bL^2_0)$, the left singular germ at $0$ 
  of the function $f(x)=\int_{-1}^1 g(x,y)dy\in C^k(\bL^1_0)$ modulo,
  respectively, functions of class $C^r$ and $W^{l,p}_{loc}$ at $x=0$. 
  Correspondingly, we set 
  $\Theta_{r,k}=\ker\theta_{r,k}$
  and $\Theta_{l,p,k}=\ker\theta_{l,p,k}$.
\end{definition}
\begin{theorem}
  \label{thm:Theta}
  The sets $\Theta_{r,k}$ and $\Theta_{l,p,k}$ satisfy the following properties
  for all $k=0,1,2,\dots,\infty$ and $p\geq1$:
  \begin{enumerate}
  \item $\Theta_{r,k}\subsetneq\Theta_{r-1,k}$ for all $1\leq r\leq k$;
  \item $\Theta_{k,k}$ contains the left singular germs of all $y$-odd $C^k$ functions;
  \item $S\fW^{l,p,k}(\bL^2_0)\subset\Theta_{l,p,k}$ for all $l=0,1,2,\dots$;
  \item $\Theta_{r,k}\subsetneq\Theta_{r,p,k}$ for all $0\leq r\leq k$.
  \end{enumerate}
\end{theorem}
\begin{proof}
  In~\cite{DeL10c} we showed that the (left singular germ of the) function 
  $$
  g(x,y)=\frac{x}{\sqrt{x^2+y^2}}
  $$ 
  provides an example of an element belonging to $\Theta_{0,\infty}$ but not 
  to $\Theta_{1,\infty}$.
  After integrating $r$ times $g$ with respect to $x$, we get concrete examples 
  of elements belonging to $\Theta_{r,\infty}$ but not to $\Theta_{r+1,\infty}$, 
  which proves point (1). Point~(2) is due to the fact that in that case
  the integral of $g$ in $y$ is zero on every interval symmetric with 
  respect to zero.

  To prove~(3) we notice that, if $g\in W^{l,p}(\bL^2_0)$,
  $$
  \|\psi\|_{W^{l,p}(\bL^1_0)}\leq
  \|g\|_{W^{l,p}(\bL^2_0)}+\|\varphi\|_{W^{l,p}(\bL^1_0)}<\infty
  $$ 
  since, by hypothesis (see Problem~1), $\varphi\in W^{l,p}(\bL^1_0)$.

  Consider now again the function $g$ used to prove point~(1).
  Then $\partial_x g(x,y)$ provides an example of function in
  $\Theta_{0,1,\infty}$ but not in $\Theta_{0,\infty}$ and 
  Point~(4) is then proved by considering, more generally, $g^{1/p}$ 
  and by integrating it with respect to $x$ as in point~(1). 
  Similar examples can also be obtained, for example, via the function 
  $g(x,y)=\left(x^2+y^2\right)^{-\alpha}$, $\alpha>0$, which belongs to 
  $\Theta_{l,p,\infty}$ for $0\leq \alpha\leq\frac{1}{p}-\frac{l}{2}$.
  Note that, for $\alpha<\frac{1}{2}$, we have also that $g\in\Theta_{0,\infty}$.
\end{proof}
We now go back to the global solution of the cohomological equation.

\noindent{\bf The Hamiltonian case.}
%
When $\xi$ is Hamiltonian with respect to the standard smooth structure, 
the projection $C^r(\cF_\xi)\to C^r(U)$, given by the restriction of a $C^r$ function
on $\cF_\xi$ to any open set $U\subset\cF_\xi$, is surjective (e.g. see~\cite{HR57}), 
so that the regularity of $\psi$ in Problem~1 only depends on the germ of 
$g$ at $(0,0)$. 

For every pair of adjacent separatrices $s_1,s_2\in\cF_\xi$ separated 
by $\gamma(s_1,s_2)\in\cG$, we denote by $a$ the common value of $F$ at every point
of the two separatrices and by $b$ the one of $G$ at every point of $\gamma$
and define the homomorphisms $\theta_{r,k,(a,b)}=\theta_{r,k}\circ T_{(a,b)}$
and $\theta_{l,p,k,(a,b)}=\theta_{l,p,k}\circ T_{(a,b)}$, where
$(T_{(a,b)}g)(x,y)=g(x-a,y-b)$.
%
\begin{theorem}
  \label{thm:Germs}
  Let $\xi$ be a planar Hamiltonian vector field, $F$ a generator of $\ker L_\xi$,
  $G$ a transversal foliation minimally separating $\cF_\xi$ and 
  $\xi'_F=\xi/L_\xi G$ the Hamiltonian vector field of $F$ 
  with respect to the symplectic form $dF\wedge dG$. Then,
  for all $k=0,\dots,\infty$, $r=0,\dots,k$, $l=0,\dots,k+3$ and $p\geq1$:
  \begin{enumerate}
  \item 
    $g\in L_{\xi'_\F}\left(C^r(\Rt)\right)\cap C^k(\Rt)$ 
    iff $[T_{(a_i,b_i)}(\Phi_\FG)_*g]_{S\fG^{r,k}(\bL^2_0)}\in \Theta_{r,k}$ 
    for all $i$;
  \item $g\in L_{\xi'_\F}\left(C^{k+1}(\Rt)\right)$ iff 
    $(\Phi_\FG)_*g$ is $C^{k+1}$ in the first variable\\ and
    $[T_{(a_i,b_i)}(\Phi_\FG)_*g]_{S\fG^{k+1,k}(\bL^2_0)}\in \Theta_{k+1,k+1}$ for all $i$;
  \item $g\in L_{\xi'_\F}(W^{l,p}_{loc}(\Rt))\cap C^k(\Rt)$ iff 
    $[T_{(a_i,b_i)}(\Phi_\FG)_*g]_{S\fW^{l,k,p}(\bL^2_0)}\in \Theta_{l,p,k}$ for all $i$.
  \end{enumerate}
\end{theorem}
\begin{proof}
  {\it 1.} If $g\in L^{(r)}_{\xi'_\F}$ then there is a $C^r$ solution $f$ to $L_\xi f=g$.
  In a normal chart $(x',y')$ of some neighborhood of any pair of 
  adjacent separatrices, the function 
  $\varphi(x')=\int_{-\epsilon}^\epsilon g(x',y')dy'\,:(-\delta,0)\to\bR$ equals 
  $f(x',\epsilon)$ modulo some function belonging to $C^r((-\delta,0])$, 
  i.e. $[\varphi]_{S\fG^r(\bL^1_0)}=0$.
  If, on the other side, $[T_{(a_i,b_i)}(\Phi_\FG)_*g]_{S\fG^r(\bL^2_0)}\in \Theta_r$ 
  for all pairs of adjacent separatrices $s_{1_i},s_{2_i}$ with transversals
  $\gamma_{1_i},\gamma_{2_i}$, then we can define any $C^r$ function of one
  of the transversals and extend the solution to the whole plane with the
  method of the characteristics. The condition 
  $[T_{(a_i,b_i)}(\Phi_\FG)_*g]_{S\fG^r(\bL^2_0)}\in \Theta_r$ grants that on every 
  separatrix we can extend the solution to a $C^r$ solution across the separatrix.
  The argument works similarly for point {\it 3} \emph{mutatis mutandis}.

  About point {\it 2}, the argument is the same but we must first prove that
  the property that $(\Phi_\FG)_*g$ is $C^{k+1}$ in the first variable does not
  depend on the particular choice of $F$ and $G$. The reason for this is that
  every other first-integral $F'$ of $\xi$ only depends on $F$, so that any
  other pair $(F',G')$, where $F'$ is a first-integral and $G'$ a transversal
  Hamiltonian for $\cF_\xi$, is such that $(F',G')=(F'(F),G'(F,G))$. 
  Hence, if $(\Phi_\FG)_*g$ is $C^{k+1}$ in the first argument for one particular
  choice of $F$ and $G$, it is so for every other choice.
\end{proof}
\begin{corollary}
  \label{thm:imagine}
  Under the hypotheses of Theorem~\ref{thm:Germs}, let $\cP$ be the set of all points 
  $(a,b)$ that separate pairs of adjacent separatrices of $(\Phi_\FG)_*\xi$ in 
  $\Phi_\FG(\Rt)$. Then the following inclusions hold:
  \begin{enumerate}
  \item $L_{\xi'_\F}(C^r(\Rt))\cap C^k(\Rt)\supset\Phi_\FG^*(C^r(\Rt)\cap C^k(\Rt\setminus\cP));$
  \item $L_{\xi'_\F}(C^{k+1}(\Rt))\supset\Phi_\FG^*(C^{k+1}(\Rt));$
  \item $L_{\xi'_\F}(W^{l,p}_{loc}(\Rt))\cap C^k(\Rt)\supset\Phi_\FG^*(W^{l,p}_{loc}(\Rt)\cap C^k(\Rt\setminus\cP))$.
  \end{enumerate}
\end{corollary}
%
Next theorem shows in particular that the solvability of the cohomological equation
is stable under small perturbations of its right hand side:
\begin{theorem}
  \label{thm:OpCl}
  Let $\xi\in\vf$ be Hamiltonian. Then:
  \begin{enumerate}
    \item $L_\xi(\Cr)\cap C^k(\Rt)$ is a clopen subset of $C^k(\Rt)$ 
      for all $r=0,\dots,k$ and $k=0,\dots,\infty$. In particular,
    $L_\xi(C^\infty(\Rt))$ is clopen in $C^\infty(\Rt)$;
    \item $L_\xi(C^{k+1}(\Rt))$, is neither open or closed in $C^k(\Rt)$ 
      for all $k=0,1,\dots$;
    \item $L_\xi(W^{l,p}_{loc})\cap C^k(\Rt)$ is a clopen subset of $C^k(\Rt)$ 
      for all $l=0,\dots,k+1$, if $p>2$, and for all $l=0,\dots,k+2$, 
      if $1\leq p\leq2$, for all $k=0,1,\dots$.
  \end{enumerate}
\end{theorem}
\begin{proof}
  {\it 1.} Set $A=L_\xi(\Cr)\cap C^k(\Rt)$ and let $g\in A$.
  Every positive function $\epsilon\in C^0(\Rt)$ defines
  a neighborhood $U_\epsilon$ of $g$ in the strong $C^k$ topology as the set 
  of all $C^k$ functions $g'$ such that 
  $$
  |g'(x,y)-g(x,y)|+\|D_{(x,y)}(g'-g)\|+\dots+\|D^k_{(x,y)}(g'-g)\|\leq\epsilon(x,y)
  $$
  for every $(x,y)\in\Rt$. If $\eta>0$ is bounded then, in any normal chart, 
  $$
  \lim_{x\to0^-}\bigg|\int_{-\eta}^\eta\partial^k_x g(x-a,y-b)dy\bigg|<\infty\hbox{ iff }
  \lim_{x\to0^-}\bigg|\int_{-\eta}^\eta\partial^k_x\left(g(x-a,y-b)+\epsilon(x,y)\right)dy\bigg|
  <\infty,
  $$
  where $(a,b)$ are the coordinates of the point that separates the two separatrices
  in the normal chart. Hence, in all normal charts, 
  $\theta_{r,k}\circ T_{(a,b)}([g])=\theta_{r,k}\circ T_{(a,b)}([g'])$ for all 
  $g'\in U_\epsilon$, namely $U_\epsilon\subset A$, namely $A$ is open.

  Now, let $\{g_n\}$ a sequence of elements of $A$ converging to $g\in\Cr$ 
  in the strong topology. Then, almost all the $g_n$ coincide with $g$ outside of
  some compact set and, therefore, in any normal chart, 
  $$
  \lim_{x\to0^-}\int_{-\epsilon}^\epsilon\partial^k_x g(x,y)dy=
  \lim_{x\to0^-}\int_{-\epsilon}^\epsilon\partial^k_x g_n(x,y)dy,
  $$
  i.e. $\theta_r([g])=\theta_r([g_n])$, for almost all $n$, namely
  $A$ is closed.

  {\it 2.} In case of $L_\xi(C^{k+1}(\Rt))$, it is enough to observe that 
  the property of being $C^{k+1}$ in the first variable in every 
  normal chart is clearly destroyed 
  by a generic $C^k$ small perturbation and is not preserved by
  $C^k$ convergence unless $k+1=k$, namely unless $k=\infty$.

  {\it 3.} The proof is the same as in point 1. The limits to the values 
  of $l$ are due to the fact that, for larger values, $W^{l,p}_{loc}$ functions
  are at least $C^{k+1}$ and therefore their set is neither open or closed
  in the $C^k$ topology.
\end{proof}
\noindent{\bf The non-Hamiltonian case.}
The method we developed for Hamiltonian vector fields is much less powerful
when $\xi$ is not Hamiltonian.
The main reason for this is that, in this case, the projection 
$C^r(\cF_\xi)\to C^r(U)$ that sends $C^r$ functions $f$ to their 
restriction $f|_U$ to an open set $U\subset\cF_\xi$ is not always 
surjective when $U$ contains separatrices~\cite{HR57}.
Hence there are constraints to the choice of 
the function $\varphi$ of Problem~1 since, if $\varphi$ does not 
extends to a global $C^r$ first integral of $\xi$, then the extension 
of the solution via the method of characteristics will sooner or
later diverge on some of the separatrices. 

We start by assuming that $\xi$ is of finite type and recall the following
property: 
\begin{proposition}
  \label{thm:phi}
  If $\xi\in\vf$ is not Hamiltonian, the differential of any generator of 
  $\ker L^{(r)}_\xi$, $r\geq1$, is zero on some of the separatrices of $\xi$. 
  Similarly, the first derivative of every solution of $L_\xi^{(r)}f=g$
  on some of the separatrices is determined by $g$ modulo constants.
\end{proposition}
\begin{proof}
Since $\xi$ is non-Hamiltonian, then the foliation $\cF_\xi$ has the following
property: there exist two adjacent separatrices $s_1$ and $s_2$ such that,
taken any two corresponding transversal segments $t_1$ and $t_2$, parametrized 
by the natural parameters $\eta_1$, $\eta_2$ with respect to the Euclidean 
metric in such a way that $\eta_i=0$ is the coordinate of $s_i\cap t_i$
and that both coordinates are positive for the points of $t_1$ and $t_2$
inside $\pi^{-1}_{\cF_\xi}(\pi^{-1}_{\cF_\xi}(t_1)\cap \pi^{-1}_{\cF_\xi}(\pi^{-1}_{\cF_\xi}(t_2)$,
then $\eta_1(\eta_2)=\eta_2^\alpha+O(\eta_2^\beta)$, with $\beta>\alpha$
and $\alpha\neq1$.
In other words, the leaves of $\cF_\xi$ approach the two separatrices
at different rates.
Assume now, for the argument's sake, that $\alpha<1$, and define a
germ of a function $\varphi(\eta_1)$ on $t_1$. Then, on $t_2$, this function
becomes $\psi(\eta_2)=\varphi(\eta_1(\eta_2))=\varphi(\eta_2^\alpha+O(\eta_2^\beta))$,
so that
$$
\frac{d\psi}{d\eta_2}\bigg|_{\eta_2}=\frac{\alpha}{\eta_2^{1-\alpha}}
\frac{d\varphi}{d\eta_1}\bigg|_{\eta_2^\alpha+O(\eta_2^\beta)}+O(\beta-1)
$$
and therefore we must have $\frac{d\varphi}{d\eta_1}\big|_{\eta_1=0}=0$
in order to be able to extend $\varphi$ to a $C^1$ function beyond $t_2$.

Regarding the second part, through the method of characteristics we have that
$$
\psi(\eta_2)=
\int\limits_{0}^{T(\eta_1(\eta_2),\eta_2)}[(\Phi_\xi^t)^*g](x(\eta_2),y(\eta_2))dt 
+\varphi(\eta_1(\eta_2)),
$$
where $\Phi_\xi^t$ is the flow of $\xi$. Hence
$$
\frac{d\psi}{d\eta_2}\bigg|_{\eta_2}=
\left[\frac{\alpha}{\eta_2^{1-\alpha}}\partial_1 T(\eta_1(\eta_2),\eta_2)
+\partial_2 T(\eta_1(\eta_2),\eta_2)\right]\left[(\Phi_\xi^{T(\eta_1(\eta_2),\eta_2)})^*g\right](x(\eta_2),y(\eta_2))+
$$
$$
+\frac{\alpha}{\eta_2^{1-\alpha}}\frac{d\varphi}{d\eta_1}\bigg|_{\eta_2^\alpha+O(\eta_2^\beta)}+O(\beta-1)
$$
and therefore we must have
$$
\frac{d\varphi}{d\eta_1}\bigg|_{\eta_1=0}+\lim_{\eta_2\to0}
\partial_1 T(\eta_1(\eta_2),\eta_2)[(\Phi_\xi^{T(\eta_1(\eta_2),\eta_2)})^*g](x(\eta_2),y(\eta_2))
=0
$$
in order to be able to extend $\varphi$ to a $C^1$ function beyond $t_2$.
\end{proof}
\begin{remark}
  Note that the derivative of $\psi(\eta_2)$ is not necessarily null 
  at $\eta_2=0$.
\end{remark}
%
Let $F\in\ker L^{(r)}_\xi$ and let $\cG=\{dG=0\}$ be any 
Hamiltonian transversal foliation which minimally separates $\cF_\xi$.
Then $\Phi_\FG=(F,G):\Rt\to\Rt$ is a $C^r$ locally injective map whose
rank is 1 on some of the separatrices. In order to make $F$ and $G$ both
regular, we switch to the $\Rt_\FG$ differential structure of the plane.
Since both $F$ and $G$ are $C^r$, then $C^r(\Rt_\FG)\subsetneq C^r(\Rt)$.
By repeating all steps as in the previous section, we get the following,
weaker, result
\begin{theorem}
  \label{thm:GermsNH}
  Let $\xi$ be a planar vector field of finite type, $F$ a generator of 
  $\ker L^{(r)}_\xi$,
  $G$ a $C^r$ transversal foliation minimally separating $\cF_\xi$ and 
  $\xi'_F=\xi/L_\xi G$ the Hamiltonian vector field of $F$ 
  with respect to the symplectic form $dF\wedge dG$. Then,
  for all $k=0,\dots,\infty$, $r=0,\dots,k$, $l=0,\dots,k+3$ and $p\geq1$:
  \begin{enumerate}
  \item $g\in L_{\xi'_\F}\left(C^r(\Rt)\right)\cap C^k(\Rt)$ if 
    $[T_{(a_i,b_i)}(\Phi_\FG)_*g]_{S\fG^{r,k}(\bL^2_0)}\in \Theta_{r,k}$ for all $i$;
  \item $g\in L_{\xi'_\F}\left(C^k(\Rt)\right)$ if 
    $(\Phi_\FG)_*g$ is $C^{k+1}$ in the first variable\\ and
    $[T_{(a_i,b_i)}(\Phi_\FG)_*g]_{S\fG^{k+1,k}(\bL^2_0)}\in \Theta_{k+1,k+1}$ for all $i$;
  \item $g\in L_{\xi'_\F}(W^{l,p}_{loc}(\Rt))\cap C^k(\Rt)$ if 
    $[T_{(a_i,b_i)}(\Phi_\FG)_*g]_{S\fW^{l,k,p}(\bL^2_0)}\in \Theta_{l,p,k}$ for all $i$.
  \end{enumerate}
\end{theorem}
In this case the conditions are sufficient but not necessary because
the cohomological equation can have $C^r(\Rt)$ solutions which do not
belong to $C^r(\Rt_\FG)$. In $\Rt_\FG$, such solution look like $C^r$ 
functions whose derivatives of order $r$ diverge on some of the separatrices
where $dF=0$ (see Proposition~\ref{thm:phi}). 
Nevertheless, Theorem~\ref{thm:GermsNH} is enough to extend
Corollary~\ref{thm:imagine} to vector fields of finite type:
\begin{corollary}
  \label{thm:imagineNH}
  Under the hypotheses of Theorem~\ref{thm:GermsNH}, let $P$ be the set of all points 
  $(a,b)$ that separate pairs of adjacent separatrices of $(\Phi_\FG)_*\xi$ in 
  $\Phi_\FG(\Rt)$. Then the following inclusions hold:
  \begin{enumerate}
  \item $L_{\xi'_\F}(C^r(\Rt))\cap C^k(\Rt)\supset\Phi_\FG^*(C^r(\Rt)\cap C^k(\Rt\setminus P));$
  \item $L_{\xi'_\F}(C^{k+1}(\Rt))\supset\Phi_\FG^*(C^{k+1}(\Rt));$
  \item $L_{\xi'_\F}(W^{l,p}_{loc}(\Rt))\cap C^k(\Rt)\supset\Phi_\FG^*(W^{l,p}_{loc}(\Rt)\cap C^k(\Rt\setminus P))$.
  \end{enumerate}
\end{corollary}

The case when separatrices are not isolated is more pathological.
We briefly discuss here only the limit case, when separatrices are 
dense on the plane in such a way that $C^1(\cF_\xi)$ contains only 
the constant functions. For such a $\xi$, every $g\in L_\xi(C^r(\Rt))$ 
determines uniquely, modulo constants, the solution of the cohomological 
equation (and, therefore, its restriction to any line). 
In particular, the method of characteristics in this context can be applied
only to $C^0$ functions since, given a transversal $t$, 
a generic $\varphi\in C^r(t)$, $r\geq1$,
will lead eventually to some ineliminable divergence. 
In $\Rt_{\Phi_\FG}$, though, both $F$ ad $G$ are regular and the method of characteristic 
can be used to study the solvability of the cohomological equation in 
$C^r(\Rt_{\Phi_\FG})$ for all values of $r$. 

Regarding the existence of more regular solutions with respect to the 
standard differential structure, we recall that, as shown
in Proposition~\ref{thm:phi}, $g$ determines the derivative of the restriction 
of the solution of the equation $L^{(r)}_\xi f=g$ on a dense set $A_t$ of any 
transversal $t$. 
\begin{definition}
  We say that a $C^r$ function $\varphi$ on $t$ is \emph{$\xi$-compatible}
  with $g$ if its derivatives coincide with those induced by $g$, via $\xi$, 
  in all points of $A_t$. 
\end{definition}
We are lead therefore to the following result:
\begin{theorem}
  Let $\xi\in\vf$ be such that $\dim C^1(\cF_\xi)=1$ and for each $s_i\in\cS_\xi$
  select a transversal $t_i$. Then $g\in L_\xi(C^r(\Rt))\cap C^k(\Rt)$ iff, 
  for each transversal $t_i$, there exists a $C^r$ function $\varphi_i$ on $t_i$ 
  $\xi$-compatible with $g$.
\end{theorem}
%
%
\section{Examples}
\label{sec:examples}
In this section we present four model examples.
\subsection{$\xi=2y\,\partial_x+(1-y^2)\,\partial_y$}
\label{ex:TodorH}
\begin{figure}
  \centering
  \includegraphics[width=5.5cm]{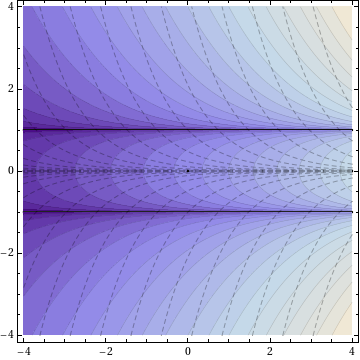}\hskip.5cm
  \includegraphics[width=5.5cm]{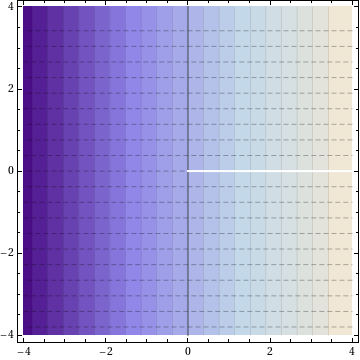}
  \caption{%
    \footnotesize [left] Foliations of the integral trajectories of 
    $\xi=2y\partial_x+(1-y^2)\partial_y$ (continuous lines) and
    $\eta=\partial_x-y\partial_y$ (dashed lines). These are tangent,
    respectively, to the level sets of the regular functions $F(x,y)=(y^2-1)e^x$ 
    and $G(x,y)=-2ye^x$. The leaves $y=\pm1$ are the only pair of inseparable 
    leaves in $\cF_\xi$, while $\cF_\eta\simeq\bR$ has no inseparable leaves.
    Note that, since $\xi$ is intrinsically Hamiltonian, $C^1(\cF_\xi)$ contains
    regular functions.
    [right] Image of $\cF_\xi$ and $\cF_\eta$ via $\Phi_\FG$. The leaves of $\cF_\xi$
    become vertical lines, those of $\cF_\eta$ horizontal ones.
    The image, under $\Phi_\FG$, of the sets $|y|<1$, $y=1$, $y=-1$ and $y=0$ 
    in the left picture are respectively represented, in the right one, 
    by the sets $(-\infty,0)\times\bR$, 
    $\{0\}\times(-\infty,0)$, $\{0\}\times(0,\infty)$ and $(-\infty,0)\times\{0\}$.
  }
  \label{fig:12}
\end{figure}
This vector field is Hamiltonian and invariant with respect to horizontal translations.
It is easy to see that its only separatrices are the straight lines $y=\pm1$. 
A regular first-integral of $L_\xi$ is the function
$F(x,y)=(y^2-1)e^x$ and a solution to the partial differential inequality $L_\xi G\neq0$
is given by $G(x,y)=-2ye^x$. 
It is easy to verify that $\cG=\{dG=0\}$ separates minimally 
$\cF_\xi$. Note that, in this particular case, $\Phi_\FG$ is globally injective.
In the normal chart $(\Rt,\Phi_\FG)$, the point that separates the two separatrices
has coordinates $(0,0)$. A straight calculation shows that
$$
\Omega_\FG = dF\wedge dG=2(1+y^2)e^{2x}dx\wedge dy
$$
and, correspondingly, 
$$
\xi'_{\F}=\Omega_\FG^{-1}(dF)=\frac{e^{-x}}{2(1+y^2)}\xi,\,\,\,
\xi'_{\G}=\Omega_\FG^{-1}(dG)=\frac{e^{-x}}{y^2+1}\left(\partial_x-y\partial_y\right)\,.
$$ 
By Proposition~6 in~\cite{DeL10c}, 
$(\Phi_\FG)_*\xi_\F=\partial_{y'}$ and $(\Phi_\FG)_*\xi_\G=\partial_{x'}$,
where we set $(x',y')=(F,G)$, and the cohomological equation writes,
in the normal chart $(\Rt\setminus[0,\infty)\times\{0\},\Phi_\FG)$, as 
$\partial_{y'}\hat f=\hat g$.

By Theorem~\ref{thm:Germs}, the equation $L_{\xi'_\F}f=g\in C^k(\Rt)$ 
has a $C^r$ solution, $r=1,\dots,k$, if and only if 
$[(\Phi_\FG)_*g]_{S\fG^{r,k}(\bL^2_0)}\in \Theta_{r,k}$, namely if and only if 
the $C^{k+1}$ function $\varphi(x')=\int_{-1}^1\hat g(x',y')dy':[-\infty,0)\to\bR$ can be 
extended to a $C^r$ function at $x'=0$. The solution is, instead, $W^{l,p}_{loc}$
if and only if $[(\Phi_\FG)_*g]_{S\fW^{l,p,k}(\bL^2_0)}\in \Theta_{l,p,k}$,
namely iff $\varphi(x')$ has a $W^{l,p}$ singularity at $x'=0$.
Below we discuss in some detail a few concrete cases. 

As shown in the proof of Theorem~\ref{thm:Theta}, the condition 
$$
|\hat g(x',y')|\leq C\left[(x')^2+(y')^2\right]^{-\alpha},\,C>0
$$
is enough to grant the existence 
of $C^0$ solutions for $\alpha<1/2$ and of $L^1_{loc}$ solutions for $1/2\leq\alpha<1$.

Consider, for example, the function
$$
\hat g(x',y')=\left[(x')^2+(y')^2\right]^{-1/4}\in C^\infty(\Rt\setminus(0,0)),
$$
so that
$$
g(x,y)=\Phi_\FG^*\hat g(x,y)=\frac{e^{-x/2}}{\sqrt{1+y^2}}\in\smooth.
$$
Then a solution to $L_{\xi'_F}g=f$ is given, in the normal chart $(x',y')$, 
by the function
$$
\hat f(x',y')= y' 
\left[1 + \frac{(y')^2}{(x')^2}\right]^{\frac{3}{4}} {}_2F_1\left(1, \frac{5}{4}, \frac{3}{2}; -\frac{(y')^2}{(x')^2}\right),
$$
where ${}_2F_1(a,b,c;z)$ is the Gaussian hypergeometric function, 
which writes as
$$
f(x,y)=
-ye^x\left[1 + \frac{4y^2}{(y^2-1)^2}\right]^{\frac{3}{4}}
{}_2F_1\left(1, \frac{5}{4}, \frac{3}{2}; -\frac{4y^2}{(y^2-1)^2}\right)
\in C^0(\Rt)
$$
in the $(x,y)$ coordinates.

On the other side, for
$$
\hat g(x',y')=\frac{1}{\sqrt{(x')^2+(y')^2}}\in C^\infty(\Rt\setminus(0,0)),
$$
namely
$$
g(x,y)=\frac{e^{-x}}{1+y^2}\in\smooth,
$$
we get 
$$
\hat f(x',y')=\ln\left(\sqrt{(x')^2+(y')^2}+y'\right),
$$
namely 
$$
f(x,y)=x+2\ln|1-y|
\in L^1_{loc}(\Rt)\cap C^\infty(\Rt\setminus S_\xi).
$$
Note that the equation $L_{\xi_F}f(x,y)=\frac{e^{-x}}{1+y^2}$ is equivalent
to $L_{\xi}f(x,y)=2$, which is why we found exactly the solution we already
discussed in Example~\ref{ex:motivations}.

More generally, the condition
$$
|\hat g(x',y')| 
\leq 
C\left[(x')^2+(y')^2\right]^{-\alpha},\,\,\,(x',y')\in U_0,\,C>0
$$
where $U_0$ is some left neighborhood of the origin, writes down, 
in the original coordinates $(x,y)$, as 
$$
|g(x,y)|\leq C e^{2\alpha x}\left[1+y^2\right]^{-2\alpha}\,,\,\,\,
(x,y)\in S_M
$$
where $S_M=(-\infty,-M)\times(-1,1)$, for some $M>0$.
Since $1+y^2$ is bounded and larger than 1 for every $y\in S_M$, 
this means that $g$ will give rise to $C^0$ solutions of $L_\xi f=g$ when 
$|g(x,y)|\leq e^{-\alpha x}$ for $\alpha<1/2$ and to $L^1_{loc}$ solutions 
for $\alpha<1$.

The latter is the same condition given in~\cite{DGK10}, Proposition~3.1, for the 
existence of $L^1_{loc}$ solutions to $L_{\xi_1}f=g$. This approach though shows that 
it is enough that this inequality be satisfied by $g$ on some neighborhood of 
$x=-\infty$ within the strip $|y|<1$ rather than on the whole plane.
\subsection{$\xi=2(2y-1)\,\partial_x+(1-y^2)\,\partial_y$}
\label{ex:TodorNH}
\begin{figure}
  \centering
  \includegraphics[width=5.5cm]{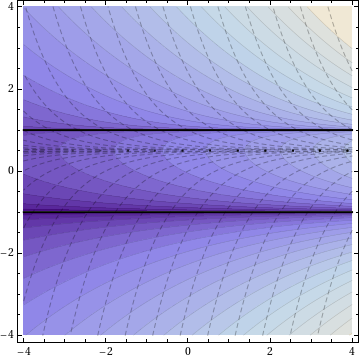}\hskip.5cm
  \includegraphics[width=5.5cm]{2}
  \caption{%
    \footnotesize [left] Foliations of the integral trajectories of 
    $\xi=2(2y-1)\,\partial_x+(1-y^2)\,\partial_y$ (continuous lines) and
    $\eta=2\partial_x+(1-2y)\partial_y$ (dashed lines). These are tangent,
    respectively, to the level sets of the function $F(x,y)=(y+1)^3(y-1)e^x$,
    whose gradient is degenerate on $y=-1$, and of the regular function 
    $G(x,y)=(2y-1)e^x$. The leaf spaces $\cF_\xi$ and $\cF_\eta$ have the same 
    topology as the corresponding ones in Fig.~\ref{fig:12} but this time
    in $C^1(\cF_\xi)$ there are no regular functions, since the derivative
    of every $C^1$ function must be zero on the leaf $y=-1$.
    [right] Image of $\cF_\xi$ and $\cF_\eta$ via $\Phi_\FG$. 
    The action of $\Phi_\FG$ is identical to the one described 
    in Fig.~\ref{fig:12}, but this time its differential on $y=-1$
    is zero, i.e. $\Phi_\FG$ is not an immersion.
  }
  \label{fig:12b}
\end{figure}
This vector field has the same separatrices as the previous one but it
is non-Hamiltonian. A smooth generator of $\ker L_\xi$ is given by
$F(x,y)=(y+1)^3(y-1)e^x$ and a regular Hamiltonian function for a
transverse foliation that minimally separates $\cF_\xi$ is given by
$G(x,y)=(2y-1)e^x$. In this case 
$$
\Omega_\FG = dF\wedge dG=2(1+y)^2(2-4y+3y^2)e^{2x}dx\wedge dy
$$
is degenerate on the separatrix $y=-1$. Correspondingly,
$$
\xi'_{\F}=\Omega_\FG^{-1}(dF)=\frac{e^{-x}}{2(2-4y+3y^2)}\xi
$$
is regular, while
$$
\xi'_{\G}=\Omega_\FG^{-1}(dG)=\frac{e^{-x}}{2(1+y)^2(2-4y+3y^2)}\left(2\partial_x+(1-2y)\partial_y\right)\,.
$$ 
diverges on $y=-1$.
Similarly, $\Phi_\FG$ maps, like in the previous example, the whole plane
injectively into $\Rt\setminus[0,\infty)\times\{0\}$, but this time $\Phi_\FG$
is not an immersion since its differential is zero on the separatrix $y=-1$.

When we restrict $\Phi_\FG$ to the set $\{|y|<1\}$, the local coordinates 
$(x',y')=(F,G)$ are well-defined and smooth we can repeat verbatim all
calculation for the explicit solutions shown in the previous example. 
This corresponds to switching differential structure in $\Rt$ and looking 
for solutions in $\Rt_{\Phi_\FG}$. We recall that, since $\xi$ is of finite type, 
$C^k(\Rt_{\Phi_\FG})\subset C^k(\Rt)$; in fact, $C^k(\Rt_{\Phi_\FG})$ is the set of
all $C^k$ functions that go to zero as $(y+1)^3$ in a neighborhood of $y=-1$.

Unlike the Hamiltonian case though, 
now the solvability conditions based on the germs of $g$ in a neighborhood 
of the separatrices are only sufficient. New solutions, not covered by the 
theorems in~\cite{DGK10}, can be found by
letting $\hat g$, in the normal chart coordinates, diverge on the 
separatrices, as long as $\Phi_\FG^*\hat g$ has the required differentiability.
Consider, for example, the case of
$$
\hat g(x',y')=\,^3\sqrt{x'}\;\left[y'-\sqrt{(x')^2+(y')^2}\right]^2
$$
This function behaves as $4\,\,^3\sqrt{x'}\,(y')^2$, of class $C^0$,
in a neighborhood of the separatrix 
$\{0\}\times(0,\infty)$, which is the image of $y=-1$, and more regularly, as 
$4\,\,^3\sqrt{(x')^{13}}\,(y')^{-2}$, of class $C^4$, in a neighborhood 
of the separatrix $\{0\}\times(-\infty,0)$, which is the image of $y=1$.
In $\Rt_{\Phi_\FG}$, therefore, $\hat g$ is of class $C^0$ and so it gives rise 
to a globally $C^0$ solution
$$
\hat f(x',y')=\,^3\sqrt{x'}\left[(x')^2 y' + \frac{2 }{3}(y')^3 - \frac{2}{3} 
\left((x')^2 + (y')^2\right)^{3/2}\right]
$$
In $\Rt$, instead, the solution is more regular:
$$
\left(\Phi_\FG^*\hat f\right)(x,y)=(1+y)(1-y)^{1/3}
\left[F^2 G + \frac{2 }{3}G^3 + \frac{2}{3}\left(F^2 + G^2\right)^{3/2}\right]e^{x/3}
$$
behaves as $(y-1)^{13/3}e^{4x/3}$ close to $y=1$, i.e. of class $C^4$, and
is smooth close to $y=1$, so we have a globally $C^4$ solution.
\subsection{$\xi=2x^2y\partial_x-\partial_y$}
\label{sec:2}
\begin{figure}
  \centering
  \includegraphics[width=5.5cm]{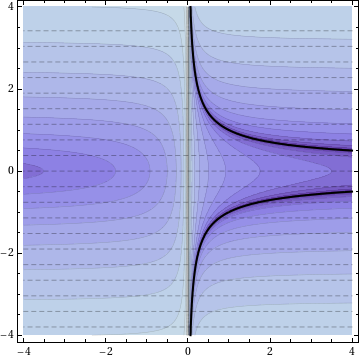}\hskip.5cm
  \includegraphics[width=5.5cm]{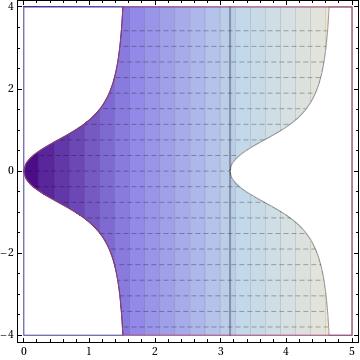}
  \caption{%
    \footnotesize [left] Foliations of the integral trajectories of 
    $\xi=2x^2y\partial_x-\partial_y$ (continuous lines) and
    $\eta=\partial_x$ (dashed lines). These are tangent,
    respectively, to the level sets of the regular functions $F(x,y)=y^2-e^{-x}$
    and $G(x,y)=y$. The leaves $y^2=e^x$ are the only pair of inseparable 
    leaves in $\cF_\xi$, while $\cF_\eta\simeq\bR$ has no inseparable leaves.
    Note that, since $\xi$ is intrinsically Hamiltonian, $C^1(\cF_\xi)$ contains
    regular functions.
    [right] Image of $\cF_\xi$ and $\cF_\eta$ via $\Phi_\FG$. 
    The leaves of $\cF_\xi$ become vertical lines, those of $\cF_\eta$ horizontal ones.
    The images of the two separatrices are the sets $\{\pi\}\times(-\infty,0)$
    and $\{\pi\}\times(0,\infty)$, the image of the line $y=0$ is the set 
    $(0,\pi)\times\{0\}$.
  }
  \label{fig:34}
\end{figure}
This vector field and the next one are not invariant with respect to any
translation and therefore are not covered by the theorems in~\cite{DGK10}.
A direct calculation shows that $L_\xi F(x,y)=0$ for the regular function
$$
F(x,y)=
\begin{cases}
  \tan^{-1}\left[y^2-\frac{1}{x}\right]&\!\!\!\!\!,\,x<0;\cr
  \pi/2&\!\!\!\!\!,\,x=0;\cr
  \tan^{-1}\left[y^2-\frac{1}{x}\right]+\pi&\!\!\!\!\!,\,x>0,\cr
\end{cases}
$$
namely $\cF_\xi$ is Hamiltonian.
The set $F^{-1}(c)$ has two connected components for $c\geq\pi$ and only one
for $c<\pi$, so the only inseparable integral trajectories of $\xi$ are the
two connected components of $F^{-1}(\pi)$, namely the curves $y=\pm1/\sqrt{x}$
(see Fig.~\ref{fig:34}, left). Since the $y$ component of $\xi$ is always
non-zero, $\cF_\xi$ is everywhere transversal to the foliation in horizontal 
straight lines and it is minimally separated by it. 
In particular, $L_\xi G(x,y)>0$ for $G(x,y)=y$ (see Fig.~\ref{fig:34}). 
In this case
$$
\Omega_{FG}=dF\wedge dG=\frac{dx\wedge dy}{(1 - x y^2)^2 + x^2}.
$$ 
and we get
$$
\xi'_F=-((1 - x y^2)^2 + x^2)\xi\,,\,\,\, 
\xi'_G=((1 - x y^2)^2 + x^2)\partial_x\,.
$$
The image of the plane via the map $\Phi_{FG}$
is the set (see Fig.~\ref{fig:34}, right)
$$
\tan^{-1}(y')^2< x'<\tan^{-1}(y')^2+\pi.
$$

All explicit calculations shown in the first example can be repeated here.
For example, this time the condition 
$$
|\hat g(x',y')| 
\leq 
C\left[(x')^2+(y')^2\right]^{-\alpha},\,\,\,(x',y')\in U_0,
$$
for the existence of regular and weak solutions of the cohomological equation
translates, in $(x,y)$ coordinates, into
$$
|g(x,y)|\leq C x^{2\alpha},\,\,\,x\in S_M,
$$
where in this case $S_M$, $M>0$, is the portion of the set $y^2-1/x<0$
contained in the half-plane $x>M$. The corresponding condition for
solutions of $L_{\xi}f=g$ is 
$$
|g(x,y)|\leq C \frac{x^{2\alpha}}{(1 - x y^2)^2 + x^2}\leq C'x^{2(\alpha-1)},\,\,\,x\in S_M,
$$
namely $L_\xi f=g$ admits $C^0$ solutions when $|g(x,y)|\leq C x^{-1-\epsilon}$ 
for some $\epsilon>0$ and $L^1_{loc}$ solutions when 
$|g(x,y)|\leq C x^{-\epsilon}$ for some $\epsilon>0$.
%
\subsection{$\xi=\displaystyle\frac{3(1+e^x y^2)^2 - e^x(6- 19 e^{ x} +22 e^{x}y^2)}
  {3 + 2 e^x(5+3y^2) + 3 e^{2 x}(1-y^2)^2 }\partial_x+\partial_y$}
%
\begin{figure}
  \centering
  \includegraphics[width=5.5cm]{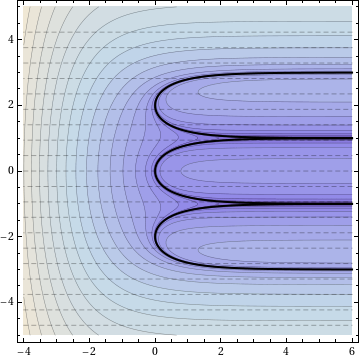}\hskip.5cm
  \includegraphics[width=5.5cm]{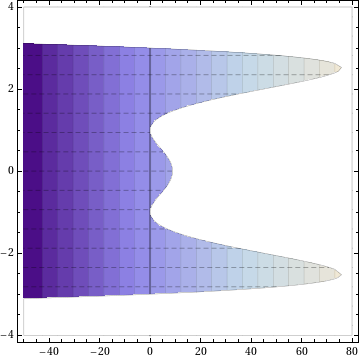}
  \caption{%
    \footnotesize [left] Foliations of the integral trajectories of 
    $\xi=\frac{3(1+e^x y^2)^2 - e^x(6- 19 e^{ x} +22 e^{x}y^2)}
    {3 + 2 e^x(5+3y^2) + 3 e^{2 x}(1-y^2)^2 }\partial_x+\partial_y$
    (continuous lines) and $\eta=\partial_x$ (dashed lines). These are tangent,
    respectively, to the level sets of the regular functions 
    $F(x,y)=(1-(y-2)^2-e^{-x})(1-y^2-e^{-x})(1-(y+2)^2-e^{-x})$ 
    and $G(x,y)=y$. The three leaves $(y\pm2)^2=1-e^{-x}$ and $y^2=1-e^{-x}$ are the 
    only separatrices of $\cF_\xi$ and are all inseparable from each other,
    while $\cF_\eta\simeq\bR$ has no inseparable leaves.
    Note that, since $\xi$ is intrinsically Hamiltonian, $C^1(\cF_\xi)$ contains
    regular functions.
    [right] Image of $\cF_\xi$ and $\cF_\eta$ via $\Phi_\FG$. The leaves of $\cF_\xi$
    become vertical lines, those of $\cF_\eta$ horizontal ones.
    The image, under $\Phi_\FG$, of the three separatrices are the sets 
    $\{0\}\times(3,1)$,
    $\{0\}\times(1,-1)$ and $\{0\}\times(-1,-3)$; the images of the lines 
    $y=\pm1$ are the sets $(-\infty,0)\times\{1\}$ and $(-\infty,0)\times\{-1\}$.
  }
  \label{fig:3seps}
\end{figure}
Also this last vector field gives rise to a Hamiltonian foliation. For example,
a generator of $\ker L_\xi^{(\infty)}$ is given by the smooth regular function
$$
F(x,y)=(1-(y-2)^2-e^{-x})(1-y^2-e^{-x})(1-(y+2)^2-e^{-x}).
$$
Using the Descartes' rule of signs it is easy to verify that, for $c>0$,
the level sets $F^{-1}(c)$ are connected while, for $c<0$, each level set
consists of three disjoint lines. The three curves in the level set $F^{-1}(0)$ 
are therefore inseparable from each other. 
Since the $y$ component of $\xi$ is always different from 0, $\xi$ is 
transversal to every horizontal line and it is easily seen that the 
foliation in horizontal lines minimally separates it (see Fig.~\ref{fig:3seps}). 
In particular, $L_\xi G=1>0$ for $G(x,y)=y$.
The image of the plane via $\Phi_\FG$ is the set
$$
x'<(1-(y'-2)^2)(1-(y')^2)(1-(y'+2)^2)
$$
(see Fig.~\ref{fig:3seps}, right). 

In this case,
$$
\Omega_\FG=dF\wedge dG=dx\wedge dy,
$$
so that 
$$
\xi_\F=\Omega_\FG^{-1}(dF)=\xi\,,\,\,\,\xi_\G=\Omega_\FG^{-1}(dG)=\partial_x
$$

The condition for the existence of solutions of $L_{\xi}f=g$
$$
|\hat g(x',y')| 
\leq 
C\left[(x')^2+(y')^2\right]^{-\alpha},\,\,\,(x',y')\in U_{\pm1},
$$
where $U_1$ and $U_{-1}$ are left neighborhoods of, respectively,
$(0,1)$ and $(0,-1)$, translates now, in $(x,y)$ coordinates, into
$$
|g(x,y)|\leq C e^{3\alpha x},\,\,\,(x,y)\in S_{M,\pm1},
$$
where in this case $S_{M,1}$ (resp. $S_{M,2}$), $M>0$, is the portion of the set
$F(x,y)<0$ contained in the half-plane $x>M$. 
Hence, for example, $L_\xi f=g$ admits $C^0$ solutions if  
$|g(x,y)|\leq C e^{(3/2-\epsilon)x},\,\,\,(x,y)\in S_{M,\pm1}$, for some $\epsilon>0$ and
$L^1_{loc}$ solutions if
$|g(x,y)|\leq C e^{(3-\epsilon)x},\,\,\,(x,y)\in S_{M,\pm1}$, for some $\epsilon>0$.

By modifying suitably this particular $F(x,y)$, it is easy to obtain examples of 
intrinsically
Hamiltonian and intrinsically non-Hamiltonian regular vector fields whose foliation 
has a single node at which concur any number of inseparable separatrices and which 
is minimally separated by the horizontal foliation.

\section*{Acknowledgments}
The author is grateful to T. Gramchev for precious suggestions that made the present article richer and more interesting.

\bibliography{refs}

\end{document}